\documentclass[12pt]{article}
\usepackage{amsmath}
\allowdisplaybreaks[4]

\usepackage[all]{xy}
\usepackage{amssymb}
\usepackage{amsthm}
\usepackage{hyperref}
\hypersetup{colorlinks=true,linkcolor=blue,citecolor=red}
\usepackage{amsmath}
\usepackage{amscd}%,enumitem
\usepackage{verbatim}
\usepackage{eurosym}
\usepackage{float}
\usepackage{color}
\usepackage{dcolumn}
\usepackage[mathscr]{eucal}
\usepackage[all]{xy}
\usepackage{hyperref}
\usepackage{mathrsfs}
\usepackage{amsmath}
\usepackage{amssymb}
\usepackage{amsfonts,ifpdf}
\usepackage{graphicx}
\usepackage{times}
\usepackage{float}
\usepackage{epstopdf}
\usepackage{cite}
\usepackage{youngtab}
\usepackage{ytableau}
\ytableausetup
{mathmode, boxsize=0.9em}

\newtheorem{thm}{Theorem}[section]

\newtheorem{lem}[thm]{Lemma}

\newtheorem{cor}[thm]{Corollary}

\newtheorem{pf}{proof}[section]
\theoremstyle{remark}

\newtheorem{remark}[thm]{Remark}

\makeatletter \@addtoreset{equation}{section} \makeatother
\makeindex 
\def\qed{\hfill \rule{4pt}{7pt}}

\def\p{{\overline{p}}}
\def\N{{\overline{N}}}

\allowdisplaybreaks

\numberwithin{equation}{section}

\begin{document}

\begin{center}
{\Large\bf Asymptotic formula for the $M_2$-ranks of overpartitions}\\ [7pt]
\end{center}

\vskip 3mm

\begin{center}
Helen W.J. Zhang$^{1,2}$ and Ying Zhong$^{1}$\\[8pt]
$^{1}$School of Mathematics\\
Hunan University\\
Changsha 410082, P. R. China\\[12pt]

$^{2}$Hunan Provincial Key Laboratory of \\
Intelligent Information Processing and Applied Mathematics\\
Changsha 410082, P. R. China\\[15pt]

Emails:  helenzhang@hnu.edu.cn,  YingZhong@hnu.edu.cn
\\[15pt]

\end{center}

\vskip 3mm

%-------------------------------------------------------------
\begin{abstract}

Let $\N_2(a,c,n)$ be the number of overpartitions of $n$ whose the $M_2$-rank is congruent to $a$ modulo $c$.
In this paper, we obtain the asymptotic formula of $\N_2(a,c,n)$ utilizing the Ingham Tauberian Theorem.
As applications, we derive inequalities concerning with $\N_2(a,c,n)$ including its strict concavity and log-concavity.

\vskip 6pt

%Partitions of integers, Combinatorial aspects of commutative algebra, Actions of groups on commutative rings; Analytic theory of partitions
\\ [7pt]
\end{abstract}

\section{Introduction}
The objective of this paper is to give an asymptotic formula for the $M_2$-rank of overpartition.
Based on this formula,
we derived several inequalities on the $M_2$-rank of overpartition such as concavity and log-concavity.

Hardy and Ramanujan \cite{Hardy-Ramanujan-1918-1} initially introduced the celebrated Circle Method (now known as  the Hardy-Ramanujan circle method) to
obtain the asymptotic formula for integer partitions originated from finding an exact formula for it.
Recall that a partition of a positive integer $n$ is a sequence of non-increasing positive integers whose sum equals $n$ and $p(n)$ is defined as the number of partitions of $n$.
Hardy and Ramanujan \cite{Hardy-Ramanujan-1918-1} proved that
\begin{align*}
p(n)\sim \frac 1{4n\sqrt{3}}e^{\pi \sqrt{\frac {2n}3}},~\text{for}~n\rightarrow\infty.
\end{align*}
Asymptotic formulas have been proven as powerful tools to gain inequalities of the partition function.
For example,
DeSalvo and Pak \cite{DeSalvo-Pak-2015} established strong log-concavity of the $p(n)$, that is, for all $n>m>1$,
\[p(n)^2>p(n-m)p(n+m)\]
from the asymptotic formula above.
See \cite{Bessenrodt-Ono-2016,Chen-Jia-Wang-2019} for further applications.

Other statistics involving partitions have been introduced since.
To interpret Ramanujan's congruences for $p(n)$ combinatorially,
Dyson \cite{Dyson-1944} introduced the rank of a partition.
This is defined as the largest part of the partition minus the number of parts.
%This partition statistic provided combinatorial interpretations of Ramanujan's congruences $p(5n+4)\equiv0\pmod5$ and $p(7n+5)\equiv0\pmod7$.
%Let $N(a,c,n)$ denote the number of partitions of $n$ whose rank is congruent to $a$ modulo $c$.
Bringmann obtained the asymptotic formula of $N(a,c,n)$ in \cite{Bringmann-2009} by the Hardy-Ramanujan circle method. Moreover, she confirmed the conjecture which was given by Andrews and Lewis \cite{Andrews-Lewis-2000},
\begin{align*}
N(0,3,n)<N(1,3,n) &~\text { if } n \equiv 0 \text { or } 2 \quad(\bmod 3), \\
N(0,3,n)>N(1,3,n) &~\text { if } n \equiv 1 \quad(\bmod 3).
\end{align*}
With the help of the Ingham Tauberian Theorem, Males \cite{Males-2021} stated another asymptotic formula of $N(a,c,n)$, that is, for fixed $0 \leq r<t$ and $t \geq 2$,
\[N(r,t,n) \sim \frac{1}{4 t n \sqrt{3}} e^{2 \pi \sqrt{\frac{n}{6}}},~\text{as}~n\rightarrow\infty.\]
The asymptotic formula yields that when $0\leq r<t$ and $t\geq 2$,
\begin{equation*}
  N(r,t,a)N(r,t,b)>N(r,t,a+b)
\end{equation*}
for sufficiently large $a$ and $b$.
Several other asymptotic formulas and inequalities for rank of a partition have also been obtained, for example \cite{Bringmann-Kane-2010,Hou-Jagadeesan-2018,Dousse-Mertens-2015,Gomez-Zhu-2021}.

Bringmann and Males illustrated two different approaches to studying the asymptotic behavior of the coefficients of a series:
One can either use the Hardy-Ramanujan circle method \cite{Bringmann-Mahlburg,Hardy-Ramanujan-1918-1}
or apply the Ingham Tauberian Theorem \cite{Ingham-1941,Jang-2017}.
%The appeal of The Ingham Tauberian theorem is that it yields asymptotics for sequences with very
%little effort, particularly in comparison to the Hardy-Ramanujan circle method,
%which typically requires modular transformations and bounds along various arcs near the complex unit circle.
In this paper,
we study an asymptotic formula for the $M_2$-rank of overpartition
in the spirit of The Ingham Tauberian Theorem.

%The first author and Liu \cite{Liu-Zhang-2021} proved that for $a,~b>1$,
%\[\p(a)\p(b)\geq \p(a+b).\]

The $M_2$-rank for overpartitions was introduced by Lovejoy \cite{Lovejoy-2008}.
Recall that an overpartition \cite{Corteel-Lovejoy-2004} of a nonnegative integer $n$ is a partition of $n$ where the first occurrence of each distinct part may be overlined. Denote $\p(n)$ by the number of overpartitions of $n$.
Let $l(\lambda)$ be the largest part of $\lambda$, $n(\lambda)$ be the number of parts of $\lambda$ and $n(\lambda_o)$  be the number of odd non-overlined parts of $\lambda$.
Then the $M_2$-rank of an overpartition $\lambda$ is defined as
\[M_2\text{-}\mathrm{rank}(\lambda)=\left\lceil\frac{l(\lambda)}{2}\right\rceil-n(\lambda)-n(\lambda_o)-\chi(\lambda),\]
where $\chi(\lambda)=1$ if $l(\lambda)$ is odd and non-overlined and $\chi(\lambda)=0$ otherwise.
Denote $\N_2(a,c,n)$  by the number of overpartitions of $n$ whose the $M_2$-rank is congruent to $a$ modulo $c$.
By investigating the asymptotic behavior of the Appell function associated with the generating function of $\N_2(a,c,n)$,
we confirm that the $M_2$-rank of overpartition fulfills the conditions in the Ingham Tauberian Theorem.
Hence we obtain the following asymptotic formula for the $M_2$-rank of overpartition.

\begin{thm}\label{N-asym}
If the odd integer $c\geq 2$, then for any $0\leq a<c$ we have
\begin{equation}\label{N-asym-eq}
  \N_2(a,c,n)\sim \frac 1c \p(n)\sim \frac 1{8cn}e^{\pi \sqrt{n}}
\end{equation}
as $n\rightarrow\infty$.
\end{thm}

We deduce strict concavity of $\N_2(a,c,n)$ by applying Theorem \ref{N-asym}.

\begin{thm}\label{con}
If the odd integer $c\geq 2$, then for any $0\leq a<c$ we have
\begin{equation*}
  \N_2(a,c,n_1)\N_2(a,c,n_2)>\N_2(a,c,n_1+n_2)
\end{equation*}
for sufficiently large $n_1$ and $n_2$.
\end{thm}

The following inequality also serves as an application of Theorem \ref{N-asym}.

\begin{thm}\label{ineq}
Let $c\geq 2$ be an odd integer and $0\leq a<c$.
Then for sufficiently large $n_1$ and $n_2$ with $n_1<n_2+1$,
we have
\begin{align*}
  \N_2(a,c,n_1)\N_2(a,c,n_2)>\N_2(a,c,n_1-1)\N_2(a,c,n_2+1).
\end{align*}
\end{thm}

Setting $n_1=n_2$ in Theorem \ref{ineq},
we show that $\N_2(a,c,n)$ is asymptotically log-concave.

\begin{cor}
If the odd integer $c\geq 2$, then for any $0\leq a<c$, we have
\begin{equation*}
  \N_2(a,c,n)^2>\N_2(a,c,n-1)\N_2(a,c,n+1)
\end{equation*}
for sufficiently large $n$.
\end{cor}

The paper is organized as follows.
In Section \ref{s-p}, we introduce the Ingham Tauberian Theorem \cite{Bringmann-Jennings-Schaffer-Mahlburg-2021,Ingham-1941} and the monotonicity of rank of overpartition obtained by Xiong and Zang \cite{Xiong-Zang-2020}.
In Section \ref{s-A}, we obtain the asymptotic behavior of the Appell function associated with the generating function of $\N_2(a,c,n)$.
In Section \ref{s-asym}, we prove our main result Theorem \ref{N-asym} and provide its applications in deriving inequalities (Theorems \ref{con} and \ref{ineq}).

\section{Preliminary}\label{s-p}
In this section, we recall some facts about the Appell function without proof.
Let the Appell function \cite{Zwegers-2002}
\begin{equation}\label{mu-def}
  \mu(u,z;\tau):=\frac{e^{\pi iu}}{\vartheta(z;\tau)}\sum_{n\in\mathbb{Z}}\frac{(-1)^n e^{\pi i(n^2+n)\tau}e^{2\pi inz}}{1-e^{2\pi in\tau}e^{2\pi iu}},
\end{equation}
where Jacobi theta function
\begin{equation}\label{theta-def}
  \vartheta(z;\tau):=\sum_{n\in\frac 12+\mathbb{Z}}e^{\pi in^2\tau+2\pi in\left(z+\frac 12\right)},
\end{equation}
with $z\in\mathbb{C}$. And $\vartheta$ satisfies
\begin{equation}\label{theta}
  \vartheta(z;\tau)=\frac {i}{\sqrt{-i\tau}}e^{\frac {-\pi iz^2}{\tau}}\vartheta\left(\frac z{\tau};-\frac 1{\tau}\right).
\end{equation}
What's more, Zwegers showed that
\begin{equation}\label{mu}
  \mu(u,v;\tau)=\frac {-1}{\sqrt{-i\tau}}e^{\frac {\pi i(u-v)^2}{\tau}}\mu\left(\frac u{\tau},\frac v{\tau};-\frac 1{\tau}\right)+\frac 1{2i}h(u-v;\tau),
\end{equation}
where $h$ is the Mordell integral
\begin{equation}\label{h-def}
  h(z;\tau):=\int_{\mathbb{R}}\frac {e^{\pi i\tau x^2-2\pi zx}}{\cosh(\pi x)}dx,
\end{equation}
which satisfied the following transformation
\begin{equation}\label{h}
  h(z;\tau)=\frac {1}{\sqrt{-i\tau}}e^{\frac {\pi iz^2}{\tau}}h\left(\frac z{\tau};-\frac 1{\tau}\right).
\end{equation}
Recently, Zwegers \cite{Zwegers-2019} introduced the Appell function of higher level %and showed that they also exhibit similar transformation formulae. We define the level $\ell$ Appell function by
\begin{equation}\label{A-def}
  A_\ell(u,v;\tau):=e^{\pi i\ell u}\sum_{n\in\mathbb{Z}}\frac{(-1)^{\ell n}q^{\frac {\ell n(n+1)}{2}}e^{2\pi inv}}{1-e^{2\pi iu}q^n},
\end{equation}
which satisfies the following transformation formulae
\begin{align}\label{A}
 A_\ell(u,v;\tau)=\sum_{k=0}^{\ell-1} e^{2\pi iuk}\vartheta\left(v+k\tau+\frac {\ell-1}2;\ell\tau\right)\mu\left(\ell u,v+k\tau+\frac{\ell-1}2;\ell\tau\right).
\end{align}

In this paper, we shall use the following theorem
\cite[Theorem 1.1]{Bringmann-Jennings-Schaffer-Mahlburg-2021} due to Ingham \cite{Ingham-1941} to obtain the asymptotic formula of $\N_2(a,c,n)$.
\begin{thm}\label{Ingham}
Let $f(q):=\sum_{n\geq0}a(n)q^n$ be a power series with weakly increasing nonnegative coefficients and radius of convergence equal to one. If there exist constants
$A>0, \lambda, \alpha \in \mathbb{R}$ such that
\begin{equation*}
  f(e^{-t})\sim\lambda t^{\alpha}e^{\frac{A}{t}} ~~~~\text{as}~ t\rightarrow0^+
\end{equation*}
and
\begin{equation}\label{b-B}
 f(e^{-\varepsilon})\ll\left|\varepsilon\right|^{\alpha}e^{\frac{A}{\left|\varepsilon\right|}} ~~~~\text{as} ~\varepsilon\rightarrow0,
\end{equation}
with $\varepsilon=x+iy~ (x, ~y\in\mathbb{R}, x>0)$ in each region of the form $|y|\leq \Delta x$ for $\Delta>0$,\\
then
\begin{equation*}
  a(n)\sim \frac{\lambda}{2\sqrt{\pi}} \frac{A^{\frac{\alpha}2+\frac 14}}{n^{\frac{\alpha}2+\frac 34}}
\end{equation*}
as $n\rightarrow\infty$.
\end{thm}
\begin{remark}\label{B-r}
Bringmann, Jennings-Schaffer and Mahlburg \cite[Remark of Theorem 1.1]{Bringmann-Jennings-Schaffer-Mahlburg-2021} noted that if in each region $|y|\leq \Delta x$
\begin{equation*}
  f(e^{-\varepsilon})\sim\lambda \varepsilon^{\alpha}e^{\frac{A}{\varepsilon}}, %~~~~\text{as}~ \varepsilon\rightarrow0,
\end{equation*}
then  the second bound in \eqref{b-B} is automatically satisfied.
\end{remark}

Let $\N_2(m, n)$ denote the number of overpartitions of $n$ with the $M_2$-rank $m$.
Xiong and Zang \cite{Xiong-Zang-2020} proved the following inequality on $\N_2(m, n)$ conjectured by Chan and Mao \cite{Chan-Mao-2014}.
We highlight this result to confirm the monotonicity of $\N_2(a,c,n)$ in Theorem \ref{Ingham}.

\begin{lem}\label{N-m}
For all $m\geq0$ and $n>0$, we have
\begin{equation*}
  \N_2(m,n)\geq\N_2(m,n-1).
\end{equation*}
\end{lem}

\section{Asymptotic behavior of $A_1(\pm z,\tau;2\tau)$}\label{s-A}

In this section, we investigate the asymptotic behavior of the Appell function $A_1(\pm z,\tau;2\tau)$ associated with the generating function of $\N_2(a,c,n)$.
This confirms the asymptotic condition in the Ingham Tauberian Theorem in our context.

\begin{thm}\label{A-thm}
Let $0<z<\frac 12$, and $\tau=\frac {i\varepsilon}{2\pi}$ with $\varepsilon=x+iy$ in each region of the form $|y|\leq \Delta x$ for $\Delta>0$ then
\begin{equation*}
  A_1(\pm z,\tau;2\tau)\rightarrow0
\end{equation*}
as $\varepsilon\rightarrow0$.
\end{thm}
\begin{proof}
Using \eqref{A}, we have
\begin{align*}
A_1(\pm z,\tau;2\tau)=\vartheta(\tau;2\tau)\mu(\pm z,\tau;2\tau).
\end{align*}
Then combining \eqref{theta} and \eqref{mu}, we find that
\begin{align}\label{A_1-theta}
A_1(\pm z,\tau;2\tau)
%&=\vartheta(\tau;2\tau)\mu(\pm z,\tau;2\tau)
%\\[5pt]
&=\frac 1{2\tau}e^{\frac {\pi i(z^2\mp2z\tau)}{2\tau}}\vartheta\left(\frac 12;-\frac 1{2\tau}\right)\mu\left(\frac {\pm z}{2\tau},\frac 12;-\frac 1{2\tau}\right)
\nonumber\\[5pt]
\quad& \quad +\frac 1{\sqrt{-8i\tau}}e^{\frac {-\pi i{\tau}}{2}}\vartheta\left(\frac 12;-\frac 1{2\tau}\right)h(\pm z-\tau;2\tau).
\end{align}
Let $S_1$ and $S_2$ denote the summations on the right-hand side of \eqref{A_1-theta}, respectively.
We first investigate the terms from $S_1$. By \eqref{mu-def}, we have
\begin{align}\label{A_1-theta-mu}
\vartheta\left(\frac 12;-\frac 1{2\tau}\right)\mu\left(\frac {\pm z}{2\tau},\frac 12;-\frac 1{2\tau}\right)
&=e^{\frac {\pm \pi i z}{2\tau}}\sum_{n\in\mathbb{Z}}\frac{(-1)^n e^{\frac{-\pi i(n^2+n)}{2\tau}}e^{\pi in}}{1-e^{\frac {-\pi in}{\tau}}e^{\frac{\pm \pi i z}{\tau}}}
\nonumber\\
\quad&=q_0^{\mp\frac{z}{2}}\sum_{n\in\mathbb{Z}}\frac{q_0^{\frac{n^2+n}{2}}}{1-q_0^{n\mp z}},
\end{align}
where $q_0:=e^{-\frac{\pi i}{\tau}}=e^{-\frac{2\pi^2}{\varepsilon}}$.

In view of \eqref{A_1-theta-mu}, $S_1$ can be rewritten as
\begin{equation*}
S_1=\frac 1{2\tau} q_0^{-\frac{1}{2}(z^2\pm z)}e^{\mp\pi iz}\sum_{n\in\mathbb{Z}}\frac{q_0^{\frac{n^2+n}{2}}}{1-q_0^{n\mp z}}.
\end{equation*}
Noting that
\begin{equation*}
  \left|e^{-\frac{1}{\varepsilon}}\right|=e^{-\frac{x}{x^2+y^2}}\leq e^{-\frac{1}{(1+\Delta^2)x}}\leq e^{-\frac{1}{(1+\Delta^2)|\varepsilon|}},
\end{equation*}
%We need to find the lowest power of $q_0$ in $S_1$, since negative power of $q_0$ give growing terms in the asymptotic limit. In other word,
so if the lowest power of $q_0$ is positive, $S_1$ tends to $0$ in our asymptotic limit.

 We first consider the plus sign, that is
 \begin{equation*}
S_1^+:=\frac 1{2\tau} q_0^{-\frac{1}{2}(z^2+ z)}e^{-\pi iz}\sum_{n\in\mathbb{Z}}\frac{q_0^{\frac{n^2+n}{2}}}{1-q_0^{n-z}}.
\end{equation*}
Considering only the inner sum without the prefactor, the $n=0$ term is
\begin{equation*}
  \frac 1{1-q_0^{-z}}=\frac {-q_0^z}{1-q_0^{z}}=-q_0^z-q_0^{2z}-\cdots,
\end{equation*}
where $z\in(0,\frac 12)$ and $\tau\in\mathbb{H}$.
For $n\geq1$, it is clear that the lowest power of $q_0$ in the inner sum is $1$.
When $n<0$, we have $n-z<0$ and hence the term
\begin{equation*}
\frac{q_0^{\frac{n^2+n}{2}}}{1-q_0^{n-z}}
=-q_0^{\frac{n^2-n+2z}{2}}\sum_{j\geq0}q_0^{(z-n)j},
\end{equation*}
with the lowest order term $-q_0^{\frac{n^2-n+2z}{2}}$. Obviously, $\frac{n^2-n+2z}{2}\geq 1+z$.
Therefore, for $z\in(0,\frac 12)$, the lowest power of $q_0$ is given by the $n=0$ term, which is
 \begin{equation*}
-\frac 1{2\tau} q_0^{-\frac{1}{2}(z^2+ z)}e^{-\pi iz}q_0^z=-\frac 1{2\tau} q_0^{-\frac{1}{2}(z^2- z)}e^{-\pi iz}.
\end{equation*}
Using the facts that $0<z<\frac 12$, $-\frac{1}{2}(z^2- z)>0$, we infer that $S_1^+\rightarrow 0$ as $\varepsilon\rightarrow0$.

We next deal with the minus sign
 \begin{equation*}
S_1^-:=\frac 1{2\tau} q_0^{-\frac{1}{2}(z^2-z)}e^{\pi iz}\sum_{n\in\mathbb{Z}}\frac{q_0^{\frac{n^2+n}{2}}}{1-q_0^{n+z}}.
\end{equation*}
Similar to the plus sign $S_1^+$, we consider the three cases about $n$. Then we find that the lowest power of $q_0$ is also given by the $n=0$ term, that is
 \begin{equation*}
\frac 1{2\tau} q_0^{-\frac{1}{2}(z^2- z)}e^{\pi iz}.
\end{equation*}
Since $-\frac{1}{2}(z^2- z)>0$, we have $S_1^-\rightarrow 0$.

We now verify the remaining term $S_2\rightarrow 0$ as well.
Applying \eqref{h}, we find that
\begin{align*}
S_2=\frac i{4\tau}e^{\frac {\pi i(z^2\mp 2z\tau)}{2\tau}}\vartheta\left(\frac 12;-\frac 1{2\tau}\right) h\left(\frac {\pm z-\tau}{2\tau};-\frac 1{2\tau}\right).
\end{align*}
On the one hand, by \eqref{theta-def}, we have
\begin{equation*}
  \vartheta\left(\frac 12;-\frac 1{2\tau}\right)=\sum_{n\in\frac 12+\mathbb{Z}}e^{-\frac {\pi in^2}{2\tau}+2\pi in}=\sum_{n\in\frac 12+\mathbb{Z}}q_0^{\frac{n^2}2}e^{2\pi in}.
\end{equation*}
It is obvious that the smallest power of $q_0$ appearing in $e^{\frac {\pi i(z^2\mp 2z\tau)}{2\tau}}\vartheta\left(\frac 12;-\frac 1{2\tau}\right)$ is given by ($n=\pm\frac 12$)
\begin{equation*}
  q_0^{-\frac{z^2}2+\frac 18}e^{\mp\pi iz}\left(e^{\pi i}+e^{-\pi i}\right).
\end{equation*}
Then the power of $q_0$ in $e^{\frac {\pi i(z^2\mp 2z\tau)}{2\tau}}\vartheta\left(\frac 12;-\frac 1{2\tau}\right)$ is positive due to $0<z<\frac 12$. Hence these terms give no contribution in the limiting situation.

On the other hand, with the definition of $h(z;\tau)$ in \eqref{h-def}, we find that
\begin{align*}
  h\left(\frac {\pm z-\tau}{2\tau};-\frac 1{2\tau}\right)
  &=\int_{\mathbb{R}}\frac {e^{-\frac {\pi ix^2}{2\tau}-2\pi\left(\frac{\pm z}{2\tau}-\frac 12\right)x}}{\cosh(\pi x)}dx\\
  \quad&=e^{\frac {-\pi iz^2}{2\tau}\pm\pi i z}\int_{\mathbb{R}\pm iz}\frac{e^{-\frac{\pi iw^2}{2\tau}+\pi w}}{\cosh\left(\pi(w\mp iz)\right)}dw,
\end{align*}
here we let $w=x\pm iz$.
Since $0<z<\frac 12$, the integrand has no poles. So we can shift the path of integration back to $\mathbb{R}$,
\begin{equation*}
  h\left(\frac {\pm z-\tau}{2\tau};-\frac 1{2\tau}\right)=e^{\frac {-\pi iz^2}{2\tau}\pm\pi i z}\int_{\mathbb{R}}\frac{e^{-\frac{\pi iw^2}{2\tau}+\pi w}}{\cosh\left(\pi(w\mp iz)\right)}dw.
\end{equation*}
Using the trivial bound
\begin{equation*}
  \frac 1{\left|\cosh\left(\pi(w\mp iz)\right)\right|}=\frac 2{\left|e^{\pi(w\mp iz)}+e^{-\pi(w\mp iz)}\right|}\leq\frac 2{\left|e^{-\pi iz}+e^{\pi iz}\right|}=\left|\sec(\pi z)\right|,
\end{equation*}
and the fact $\int_{\mathbb{R}}e^{-t^2}dt=\sqrt{\pi}$,
we find that
\begin{equation*}
  \left|h\left(\frac {\pm z-\tau}{2\tau};-\frac 1{2\tau}\right)\right|\leq \left|\sec(\pi z)\right|\exp\left(-\frac{\pi^2z^2}{\varepsilon}+\frac {\varepsilon}4\right)\pi^{-\frac 12}\varepsilon^{\frac 12},
\end{equation*}
where $\tau=\frac {i\varepsilon}{2\pi}$.
Hence, we have
\begin{equation*}
  \left|\frac{i}{4\tau} h\left(\frac {\pm z-\tau}{2\tau};-\frac 1{2\tau}\right)\right|\leq \left|\sec(\pi z)\right|\exp\left(-\frac{\pi^2z^2}{\varepsilon}+\frac {\varepsilon}4\right)\frac{\pi^{\frac 12}}{2\varepsilon^{\frac 12}}.
\end{equation*}
It is easy to see that as $\varepsilon\rightarrow0$ this contribution vanishes. This completes the proof.
\end{proof}

\section{Asymptotic formula and inequalities}\label{s-asym}

In this section, we employ the Ingham Tauberian Theorem on the generating function of $\N_2(a,c,n)$ to deduce its asymptotic formula.
Several inequalities concerning on $\N_2(a,c,n)$ including its strict concavity and log-concavity are obtained based on this formula.

We basically handle with the generating function of $\N_2(m, n)$.
Lovejoy \cite{Lovejoy-2008} found that
\begin{align}\label{N-gener}
 \overline{R}_2(\zeta,q)&=\sum_{n=0}^{\infty} \sum_{m=-\infty}^{\infty}\N_2(m, n)\zeta^mq^n=\sum_{n=0}^{\infty}\frac{(-1;q)_{2n}q^n }{(\zeta q^2;q^2)_n(\zeta^{-1}q^2;q^2)_n}\notag \\
 \quad &=\frac{(-q;q)_\infty}{(q;q)_\infty}\sum_{n=-\infty}^{\infty}\frac{(1-\zeta)(1-\zeta^{-1})(-1)^n q^{n^2+2n}}{(1-\zeta q^{2n})(1-\zeta^{-1}q^{2n})},
\end{align}
where
\begin{equation*}
  (a;q)_\infty=\prod_{n=0}^{\infty}(1-aq^n)~~\text{and}~~(a;q)_n=\frac {(a;q)_\infty}{(aq^n;q)_\infty}.
\end{equation*}
In particular,
\begin{equation*}
  \overline{R}_2(1,q)=\sum_{n=0}^{\infty}\p(n)q^n.
\end{equation*}

The following lemmas play central roles in the proof of Theorem \ref{N-asym}.
\begin{lem}
For integer $0\leq a<c$  and odd integer $c$, we have
\begin{equation}\label{N}
  \sum_{n=0}^{\infty} \N_2(a,c,n)q^n=\frac 1c\left[ \sum_{n=0}^{\infty}\p(n)q^n+\sum_{j=1}^{\frac{c-1}2}\left(\zeta_c^{-aj}+\zeta_c^{aj}\right)\overline{R}_2(\zeta_c^{j},q)\right],
\end{equation}
where $\zeta_c=e^{{2\pi i}/c}$ and $q=e^{2\pi i\tau}, \tau\in\mathbb{H}$.
\end{lem}

To verify \eqref{N}, we need to show that
\begin{equation}\label{N-1}
  \sum_{n=0}^{\infty} \N_2(a,c,n)q^n=\frac 1c \sum_{n=0}^{\infty}\p(n)q^n+\frac 1c\sum_{j=1}^{c-1}\zeta_c^{-aj}\overline{R}_2(\zeta_c^{j},q),
\end{equation}
where $\zeta$ is the root of unity.

By \eqref{N-gener}, we notice that the right hand side of \eqref{N-1} can be expanded as
\begin{align*}
 \frac 1c \sum_{n=0}^{\infty}\p(n)q^n+\frac 1c\sum_{j=1}^{c-1}\zeta_c^{-aj}\overline{R}_2(\zeta_c^{j},q)
 &=\frac 1c \sum_{j~~(\text{mod}~c)} \sum_{n=0}^{\infty} \sum_{m=-\infty}^{\infty} \N_2(m,n)\zeta_c^{mj}\zeta_c^{-aj}q^n
 \\
 &=\frac 1c \sum_{n=0}^{\infty} \sum_{m=-\infty}^{\infty} \N_2(m,n)\left(\sum_{j~~(\text{mod}~c)}\zeta_c^{mj-aj}\right)q^n.
\end{align*}
By the orthogonality of the roots of unity
\[\sum_{j\pmod{c}}\zeta_{c}^{rj}=\left\{\begin{array}{ll}
0, &r\not\equiv0\pmod{c},\\
c, &r\equiv0\pmod{c},
\end{array}\right.\]
we obtain
\begin{align*}
\frac{1}{c}\sum_{n=0}^\infty \overline{p}(n)q^n
+\frac{1}{c}\sum_{j=1}^{c-1}\zeta^{-aj}\overline{R}_2\left(\zeta_c^j;q\right)
&=\sum_{n=0}^\infty\sum_{m=-\infty\atop m\equiv a\pmod{c}}^{\infty}\N_2(m,n)q^n
\nonumber\\[3pt]
&=\sum_{n=0}^\infty\N_2(a,c,n)q^n,
\end{align*}
which gives \eqref{N-1}. Since $\overline{R}_2\left(\zeta;q\right)=\overline{R}_2\left(\zeta^{-1};q\right)$, then applying \eqref{N-1}, we can get \eqref{N}.
\qed

Then we prove an identity on the $M_2$-rank generating function $\overline{R}_2(\zeta,q)$.
\begin{lem}
Let $\tau\in\mathbb{H}, z\in \mathbb{R}, q:=e^{2\pi i\tau}$ and $\zeta:=e^{2\pi iz}$. We have
\begin{align}\label{R}
 \overline{R}_2(\zeta,q)=\frac{(-q;q)_\infty}{(q;q)_\infty}\frac {(1-\zeta)}{1+\zeta}\left\{\zeta^{\frac 12}A_1(z,\tau;2\tau)-\zeta^{-\frac 12}A_1(-z,\tau;2\tau)\right\}.
\end{align}
\end{lem}

\begin{proof}
One can easily derive that
\begin{equation}\label{zeta}
  \frac {\zeta}{1-\zeta q^{2n}}-\frac {\zeta^{-1}}{1-\zeta^{-1} q^{2n}}=\frac {\zeta-\zeta^{-1}}{\left(1-\zeta q^{2n}\right)\left(1-\zeta^{-1} q^{2n}\right)}.
\end{equation}
Substituting \eqref{zeta} into \eqref{N-gener}, we obtain that
\begin{align*}
 \overline{R}_2(\zeta,q)&=\frac{(-q;q)_\infty}{(q;q)_\infty}\sum_{n=-\infty}^{\infty} (-1)^n q^{n^2+2n}\cdot\left( \frac {\zeta}{1-\zeta q^{2n}}-\frac {\zeta^{-1}}{1-\zeta^{-1} q^{2n}}\right)\frac{(1-\zeta)(1-\zeta^{-1})}{\zeta-\zeta^{-1}}\\
 \quad&=\frac{(-q;q)_\infty}{(q;q)_\infty}\frac {1-\zeta}{1+\zeta}\left(\sum_{n=-\infty}^{\infty} \frac{(-1)^n q^{n^2+2n}\zeta}{1-\zeta q^{2n}}-\sum_{n=-\infty}^{\infty} \frac{(-1)^n q^{n^2+2n}\zeta^{-1}}{1-\zeta^{-1} q^{2n}}\right)\\
 \quad&=\frac{(-q;q)_\infty}{(q;q)_\infty}\frac {(1-\zeta)}{1+\zeta}\left\{\zeta^{\frac 12}A_1(z,\tau;2\tau)-\zeta^{-\frac 12}A_1(-z,\tau;2\tau)\right\}.
\end{align*}
This completes the proof.
\end{proof}

Now we are in a position to prove Theorem \ref{N-asym}.

\begin{proof}[Proof of Theorem \ref{N-asym}]
We can rewrite $\N_2(a,c,n)$ as
\begin{equation*}
\N_2(a,c,n)=\sum_{k\in\mathbb{Z}}\N_2(a+kc,n),
\end{equation*}
which is a finite sum, since for $|a+kc|>n$ we have $\N_2(a+kc,n)=0$.

By Lemma \ref{N-m}, we derive that
for all non-negative integers $m$ and positive integers $n$,
\begin{equation*}
  \N_2(a,c,n)\geq\N_2(a,c,n-1).
\end{equation*}

We are therefore in the situation where we may apply Theorem \ref{Ingham}.
By Remark \ref{B-r}, we only need to investigate the asymptotic behavior
\begin{equation*}
 \lim_{\varepsilon\rightarrow0} \sum_{n=0}^{\infty} \N_2(a,c,n)e^{-\varepsilon n}.
\end{equation*}
Combining \eqref{N} and \eqref{R}, we obtain that
\begin{align*}
\sum_{n=0}^{\infty} \N_2(a,c,n)q^n=&\frac 1c \sum_{n=0}^{\infty}\p(n)q^n \Bigg(1+\sum_{j=1}^{\frac {c-1}2}\left(\zeta_c^{-aj}+\zeta_c^{aj}\right)\frac {(1-\zeta_c^j)}{1+\zeta_c^j}
\\
&\times \left(\zeta_{2c}^{j}A_1(\frac jc,\tau;2\tau)-\zeta_{2c}^{-j}A_1(-\frac jc,\tau;2\tau)\right)\Bigg).
\end{align*}
Let $\tau=\frac {i\varepsilon}{2\pi}$ with $\varepsilon=x+iy$ in each region $|y|\leq \Delta x$ for $\Delta>0$ and consider $\varepsilon\rightarrow0$. Since $0<\frac jc<\frac 12$, we can know that the term in the biggest brackets is asymptotically equal to $1$ in this limit by Theorem \ref{A-thm}. Hence we have
\begin{equation*}
 \lim_{\varepsilon\rightarrow0} \sum_{n=0}^{\infty} \N_2(a,c,n)e^{-\varepsilon n}
 %=\lim_{\varepsilon\rightarrow0^+}\frac 1c \sum_{n=0}^{\infty}\p(n)e^{-\varepsilon n}
 \sim \frac {\varepsilon^{\frac 12}}{2c\sqrt{\pi}}e^{\frac {\pi^2}{4\varepsilon}}
\end{equation*}
in these regions of restricted angle.
Then applying Theorem \ref{Ingham} we see that for $n\rightarrow\infty$
\begin{equation*}
  \N_2(a,c,n)\sim \frac 1{8cn}e^{\pi \sqrt{n}}.
\end{equation*}
By the following well-known fact (see, e.g., \cite{Hardy-Ramanujan-1918})
\begin{equation}\label{asym-p}
  \p(n)\sim \frac 1{8n}e^{\pi \sqrt{n}},~\text{for}~n\rightarrow\infty,
\end{equation}
we arrive at \eqref{N-asym-eq}, which completes the proof.
\end{proof}
%\section{Proof of Theorems \ref{con} and \ref{ineq}}\label{Sec-5}

We prove Theorems \ref{con} and \ref{ineq} as applications of Theorem \ref{N-asym}.

\begin{proof}[Proof of Theorem \ref{con}]
 Consider the ratio
\begin{equation*}
  \frac {\N_2(a,c,n_1)\N_2(a,c,n_2)}{\N_2(a,c,n_1+n_2)}
\end{equation*}
as $n_1, n_2 \rightarrow\infty$. By Theorem \ref{N-asym} we have
\begin{align*}
\frac {\N_2(a,c,n_1)\N_2(a,c,n_2)}{\N_2(a,c,n_1+n_2)}
\sim \frac{n_1+n_2}{8cn_1 n_2}\cdot\frac{e^{\left(\pi \sqrt{n_1}+\pi \sqrt{n_2}\right)}}{e^{\pi \sqrt{n_1+n_2}}}
>1
\end{align*}
as $n_1, n_2 \rightarrow\infty$, which completes the proof.
\end{proof}

\begin{proof}[Proof of Theorem \ref{ineq}]
 Similar to the proof of Theorem \ref{con}, we have to consider the ratio
 \begin{equation*}
  \frac {\N_2(a,c,n_1)\N_2(a,c,n_2)}{\N_2(a,c,n_1-1)\N_2(a,c,n_2+1)}.
\end{equation*}
For $n_1<n_2+1$, by Theorem \ref{N-asym}, as $n_1, n_2\rightarrow\infty$ we have
\begin{align*}
\frac {\N_2(a,c,n_1)\N_2(a,c,n_2)}{\N_2(a,c,n_1-1)\N_2(a,c,n_2+1)}
&\sim\frac{(n_1-1)(n_2+1)}{n_1 n_2}\cdot\frac{e^{\left(\pi \sqrt{n_1}+\pi \sqrt{n_2}\right)}}{e^{\left(\pi \sqrt{n_1-1}+\pi \sqrt{n_2+1}\right)}}
>1.
\end{align*}
This completes the proof.
%The case of $n_1>n_2+1$ is opposite, i.e.
%\begin{align*}
% \frac {\N_2(a,c,n_1)\N_2(a,c,n_2)}{\N_2(a,c,n_1-1)\N_2(a,c,n_2+1)}
%<1
%\end{align*}
%for sufficiently large $n_1$ and $n_2$.
\end{proof}

\vspace{0.5cm}
 \baselineskip 15pt
{\noindent\bf\large{\ Acknowledgements}} \vspace{7pt} \par
This work was supported by the National Natural Science Foundation of China (Grant Nos. 12001182 and 12171487),  the Fundamental Research Funds for the Central Universities (Grant No. 531118010411) and Hunan Provincial Natural Science Foundation of China (Grant No. 2021JJ40037).

\end{document}